\documentclass[12pt]{amsart}
\usepackage{amssymb}
\usepackage{amsfonts}
\usepackage{amssymb,latexsym}
\usepackage{enumerate}
\usepackage{mathrsfs}
\usepackage{url}
\usepackage{rotating}
\usepackage{lscape}
\allowdisplaybreaks

\makeatletter
\@namedef{subjclassname@2020}{%
	\textup{2020} Mathematics Subject Classification}
\makeatother

\newtheorem{theorem}{Theorem}[section]
\newtheorem{lemma}[theorem]{Lemma}

\newtheorem{corollary}[theorem]{Corollary}

\theoremstyle{definition}
\newtheorem{remark}[theorem]{Remark}
\newtheorem{definition}[theorem]{Definition}

\usepackage{array}
\def\bar{\tilde}
\def\ga{\gamma}

\numberwithin{equation}{section} 

\begin{document}

\def\t{\widetilde}
\def\tilde{\widetilde}
\def\ga{\gamma}
\def\d{{\rm d}}

\title[Uniqueness for $p$-adic zeta equations]
{Infinite-order $p$-adic differential equations for Hurwitz-type Euler zeta functions: Uniqueness and higher-dimensional extensions}

\author{Su Hu}
\address{Department of Mathematics, South China University of Technology, Guangzhou, Guangdong 510640, China}
\email{mahusu@scut.edu.cn}

\author{Min-Soo Kim}
\address{Department of Mathematics Education, Kyungnam University, Changwon, Gyeongnam 51767, Republic of Korea}
\email{mskim@kyungnam.ac.kr}


\begin{abstract}
We generalize the one-dimensional infinite-order linear differential equation satisfied by the $p$-adic Hurwitz-type Euler zeta function (Abh. Math. Semin. Univ. Hambg. 91: 117--135, 2021) to $n$ variables. By encoding the differential operator as a convolution with a distribution kernel, we introduce partial zeta functions and partial operators indexed by subsets of $\{1,\dots,n\}$. A tensor product expansion of the distribution kernel is established, whose Möbius inversion via the binomial theorem yields the higher-dimensional analogue of the original equation in the form of an alternating-sum identity, reducing to the one-dimensional case when $n=1$. Convergence of the resulting series is confirmed by non-Archimedean estimates. As a second main contribution, we prove that, under the condition $|a|_p > 2p^{1/(p-1)}$, within the Banach space of bounded analytic functions on $\mathbb Z_p$, the shifted $p$-adic Hurwitz-type Euler zeta function is the \emph{unique} solution to the one-dimensional equation. This uniqueness result establishes that the infinite-order $p$-adic differential equation admits at most one bounded analytic solution, and together with the existence theorem, it characterizes the shifted $p$-adic Hurwitz-type Euler zeta function uniquely. This sharply distinguishes it from the complex setting, where the operator series diverges on the  Hurwitz zeta function, and the formal kernel fails to converge in the usual spaces of analytic test functions.
\end{abstract}
\maketitle

\def\ord{\text{ord}_p}
\def\ordt{\text{ord}_2}
\def\o{\omega}
\def\la{\langle}
\def\ra{\rangle}
\def\Log{{\rm Log}\, \Gamma_{p,N}^*}
\def\ov{\bar}

\section{Introduction}\label{Section 1}

In 1900, Hilbert \cite{HilbertProb} asserted that $\zeta(s)$ is not a solution of any algebraic ordinary differential equation on its region of analyticity; this was later proved rigorously by Ostrowski \cite{Ostrowski1920}, who showed more generally that a Dirichlet series satisfying an algebraic differential equation must be of a very restricted form. This did not, however, foreclose the possibility of \emph{non-algebraic} or \emph{infinite-order} differential equations.

In 2015, Van Gorder \cite{VanGorder2015} took up this challenge and showed that the Riemann zeta function formally satisfies an infinite-order linear differential equation
\[
T[\zeta(s) - 1] = \frac{1}{s-1},
\]
where $T = \sum_{n=0}^{\infty} P_n(s) \exp(nD_s)$ is a differential operator of infinite order. Shortly after, Prado and Klinger-Logan \cite{Prado2020} extended this result to the ordinary Hurwitz zeta function $\zeta(s,a)$ and to Dirichlet $L$-functions. They proved that, in the complex plane, the series defining the operator $T$ applied to $\zeta(s,a)$ \emph{diverges everywhere}, so that the operator cannot be applied to it in the complex analytic sense.

In 2021, we \cite{HuKim2021} replaced the complex variable $s$ with a $p$-adic variable and considered the alternating Hurwitz-type Euler zeta function
\[
\zeta_{p,E}(s,a) = \sum_{n=0}^{\infty} \frac{(-1)^n}{(n+a)^s},
\]
which is the $p$-adic interpolation of the classical Euler zeta function (see \cite{KS} for the definition and properties of $\zeta_{p,E}(s,a)$). By exploiting the non-Archimedean inequality and the properties of the Teichmüller character $\omega_v(a)$, we proved that the $p$-adic analogue $T_p^a$ of Van Gorder's operator converges in the $p$-adic topology. The main theorem states that, for $s \in \mathbb{Z}_p$ with $s \neq 1$ and $a \in \mathbb{C}_p$ with $|a|_p > 1$,
\begin{equation}\label{1.1}
T_p^a\left[\zeta_{p,E}(s,a) - \langle a \rangle^{1-s}\right]
=
\frac{1}{s-1}\left( \langle a-1 \rangle^{1-s} - \langle a \rangle^{1-s} \right),
\end{equation}
where $\langle a \rangle$ denotes the projection onto the principal units (see \cite[Theorem 3.5]{HuKim2021}). In contrast with the complex case, due to the non-Archimedean property, the operator $T_p^a$ applied to $\zeta_{p,E}(s,a)$ converges $p$-adically in the region $s\in \mathbb{Z}_p$ with $s\neq 1$ and $a\in K$ with $|a|_p>1$, where $K$ is any finite extension of $\mathbb{Q}_p$ with ramification index over $\mathbb{Q}_p$ less than $p-1$ (see \cite[Corollary 3.8]{HuKim2021}).

In a companion paper \cite{HuKim2024}, we considered the function-field analogue of this problem. Let $K=\mathbb{F}_{q}(T)$ be the rational function field, $K_{\infty}=\mathbb{F}_{q}((\frac{1}{T}))$ be the completion of $K$ at the infinite place $\infty=\left(\frac{1}{T}\right)$ and $K_{\infty}^{*}=K_{\infty}\setminus\{0\}$ be the multiplicative group of the field $K_{\infty}$. By introducing a Hurwitz-type refinement of the Goss zeta function $\zeta_\infty(s_0, s, a, n)$ and an infinite-order linear difference operator $L$ (see \cite[(2.13)]{HuKim2024}), we established the difference equation
\[
L\left[ \zeta_\infty\left( \frac{1}{T}, s, a, 0 \right) \right]
=
\sum_{\gamma \in \mathbb{F}_q} \frac{1}{\langle a + \gamma \rangle^s}
\]
(see \cite[Theorem 2.1]{HuKim2024}). This result may be viewed as the positive-characteristic analogue of (\ref{1.1}), where differential operators are replaced by difference operators to accommodate the discrete component of the Goss plane $\mathbb{S} = K_\infty^* \times \mathbb{Z}_p$ (see \cite[Definition 2.1]{Goss}).

In this paper, we reinterpret \eqref{1.1} using the language of $p$-adic distributions, following the framework developed in \cite{Khrennikov1991superspace, Khrennikov1991, Khrennikov1994book}. This reformulation has two advantages: it reveals the equation as a convolution with a distribution kernel, and it extends naturally to several variables via tensor products. The key observation is that the exponential translation operator $\exp_p(nD_s)$ in \cite{HuKim2021} is the translation operator on the space of $p$-adic analytic functions, which in distribution theory is represented as convolution with a shifted Dirac delta.
Hence, the entire operator $T_p^a$ can be encoded as a single distribution kernel
\[
K_s(t) := \sum_{n=0}^{\infty} P_{p,n}^a(s) \, \delta(t - s - n),
\]
where the coefficients $P_{p,n}^a(s)$ are those defining $T_p^a$ (see Section~3). Consequently, (\ref{1.1}) is equivalent to the convolution equation
\[
(K_s * F)(s) = R(s), \quad F(s) := \zeta_{p,E}(s,a) - \langle a \rangle^{1-s},
\]
in the space of $p$-adic distributions. This reformulation reveals the equation as a linear functional equation, and the non-Archimedean topology ensures convergence of the kernel series whenever the coefficients decay $p$-adically, a phenomenon with no complex analogue.

The contrast with the Archimedean setting is instructive. There, the coefficients $p_n(s)$ do not tend to zero; indeed, for positive integer $s$, the series $\sum_{n=0}^\infty p_n(s)$ diverges \cite[Lemma~9]{Prado2020}. Consequently, the formal kernel
$\sum_{n=0}^\infty p_n(s)\delta(t-s-n)$
cannot be defined as a distribution in the usual spaces of analytic test functions, since the pairing with the constant function $1$ is not convergent. In the $p$-adic norm, however, the same coefficients tend to zero, so that the kernel becomes a well-defined $p$-adic distribution and the convolution equation is meaningful (see Remark \ref{Remark3.7}).

Beyond this reformulation, the distributional framework also allows us to address a fundamental question that remained open in \cite{HuKim2021}: is the solution to \eqref{1.1} unique? In Section~3.3 we prove that, under the condition
\[
|a|_p > 2\,p^{1/(p-1)},
\]
within the Banach space $\mathcal B$ of bounded analytic functions on $\mathbb Z_p$, the homogeneous equation $T_p^a[H]=0$ admits only the trivial solution $H\equiv0$. This condition is stronger than the convergence condition and ensures that the series defining the homogeneous operator has norm strictly less than one, making it a contraction on $\mathcal B$. Consequently, for such $a$, the shifted $p$-adic Hurwitz-type Euler zeta function
\[
F(s)=\zeta_{p,E}(s,a)-\langle a\rangle^{1-s}
\]
is the \emph{unique} bounded analytic solution to the non-homogeneous equation. This uniqueness result establishes that the infinite-order $p$-adic differential equation admits at most one bounded analytic solution, and together with the existence theorem, it characterizes the shifted $p$-adic Hurwitz-type Euler zeta function uniquely. This sharply distinguishes it from the complex setting, where the series defining the analogous infinite-order operator diverges when applied to the Hurwitz zeta functions (see \cite[Theorem 8]{Prado2020}).

Once phrased distributionally, a natural higher-dimensional generalization emerges. The original result is confined to a single variable $s$. However, the distribution kernel construction extends directly to multiple variables by taking tensor products of one-dimensional kernels. The multi-variable setting is not an artificial extension, because when expanding the product of translated zeta functions, each variable can be either ``activated'' (involving integration) or left at a constant background term. This selection is naturally encoded by subsets $I\subseteq\{1,\dots,n\}$, leading to a decomposition into $2^n$ partial operators. This tensor product expansion is formalized in Theorem~\ref{prop:tensor_main}(i).

Thus we may lift \eqref{1.1} to $n$ variables. Defining  the higher-dimensional $p$-adic Hurwitz-type Euler zeta function by
\[
\zeta_{p,E}^{(n)}(\mathbf S, a)
:= \int_{\mathbb Z_p^n} \prod_{i=1}^n \langle a + t_i \rangle^{1-s_i} \, d\mu_{-1}^{\otimes n},
\]
where $\mathbf S = (s_1,\dots,s_n) \in \mathbb Z_p^n$. For each subset $I\subseteq\{1,\dots,n\}$, we introduce a partial zeta function $Z_I(\mathbf S,a)$ that integrates only over variables in $I$ (see Definition~\ref{def:highzeta}), and a corresponding partial operator $T_p^{a,I}$ that acts as the one-dimensional $p$-adic Van Gorder operator on variables in $I$ and as multiplication by constants on the complement (see Definition~\ref{def:highop}). Our main result is the following higher-dimensional analogue of \eqref{1.1}:

\begin{equation}\label{main}
\sum_{I\subseteq\{1,\dots,n\}} (-1)^{n-|I|}
T_p^{a,I}\big[Z_I(\mathbf S,a)\big]
=
\prod_{i=1}^n
\left[
\frac{1}{s_i-1}
\left( \langle a-1\rangle^{1-s_i} - \langle a\rangle^{1-s_i} \right)
\right].
\end{equation}

This identity holds for all $s_i \in \mathbb{Z}_p$ with $s_i \neq 1$ for each $i$, and $a\in\mathbb C_p$ with $|a|_p>1$, in the sense of convergent $p$-adic series (see Theorem~\ref{prop:tensor_main}(ii)). It reduces exactly to the original one-dimensional theorem \eqref{1.1} when $n=1$, since the alternating sum then has only two terms ($I=\emptyset$ and $I=\{1\}$); see Remark~\ref{rem:degeneration}. For $n>1$, the right-hand side becomes a product of one-dimensional factors, reflecting the independent action of the operator on each variable. The convergence of the alternating sum is guaranteed by the non-Archimedean decay of the coefficients; precise conditions are given in Theorem~\ref{thm:conv}, where it is assumed that \(a\) lies in a finite extension \(K/\mathbb{Q}_p\) with ramification index \(e<p-1\). Under this hypothesis the one-dimensional estimate
\[
\left| \frac{\prod_{j=1}^{N-1}(s_i-1+j)}{N! \, \omega_v(a)^N} \right|_p \le p^{N v_p(a)} N
\]
ensures the required decay. The uniformity of convergence on compact subsets of $\mathbb Z_p^n$ is established in Theorem~\ref{thm:conv}.

The paper is organized as follows. Section \ref{Section 2}  collects the necessary preliminaries on \(p\)-adic zeta functions and \(p\)-adic distributions. Section \ref{Section 3}  reformulates the one-dimensional equation as a convolution with a distribution kernel and proves the uniqueness of bounded analytic solutions under the condition $|a|_p > 2p^{1/(p-1)}$. Section \ref{Section 4} develops the higher-dimensional generalization: we define partial zeta functions and partial operators, establish the tensor product expansion and its M\"obius inversion (Theorem~\ref{prop:tensor_main}), and prove the convergence of the alternating sum (Theorem~\ref{thm:conv}).

\section{Preliminaries}\label{Section 2}
\label{sec:prelim}

In this section, we collect the necessary background materials from two distinct but compatible theories: the $p$-adic analytic number theory of Hurwitz-type zeta functions (following our previous work \cite{HuKim2021, KS}), and the theory of $p$-adic generalized functions on superspace (following Khrennikov \cite{Khrennikov1991superspace, Khrennikov1991}).

\subsection{The fermionic $p$-adic integral and the Hurwitz-type Euler zeta function}
\label{subsec:fermionic_int}

We begin by recalling the decomposition of the multiplicative group $\mathbb{C}_p^\times$ that is central to the definition of $p$-adic zeta functions.

For $a \in \mathbb{Z}_p$ with $p \nmid a$, there exists a unique $(p-1)$-th root of unity $\omega(a) \in \mathbb{Z}_p$ such that 
\[
a \equiv \omega(a) \pmod{p}.
\]
This $\omega$ is the \emph{Teichm\"uller character}. We define the projection onto the principal units by $\langle a \rangle := \omega^{-1}(a) a$, so that $\langle a \rangle \equiv 1 \pmod{p}$.

Following Tangedal and Young \cite{Tangedal2011}, we extend these definitions to $\mathbb{C}_p^\times$. Let $\mu$ be the group of roots of unity of order prime to $p$. For $a \in \mathbb{C}_p$ with $|a|_p = 1$, there is a unique $\hat{a} \in \mu$ such that $|a - \hat{a}|_p < 1$; this $\hat{a}$ is the \emph{Teichm\"uller representative}. For general $a \in \mathbb{C}_p^\times$, write $a = p^r u$ with $r \in \mathbb{Q}$ and $|u|_p = 1$, and define $\hat{a} := \hat{u}$. Then we set
\begin{equation}
\omega_v(a) := p^{v_p(a)} \hat{a}, \quad \langle a \rangle := \frac{a}{\omega_v(a)}.
\end{equation}\label{omegav}
This yields the internal product decomposition
\begin{equation}
\mathbb{C}_p^\times \simeq p^{\mathbb{Q}} \times \mu \times D, \quad D := \{ a \in \mathbb{C}_p : |a-1|_p < 1 \},
\end{equation}
via $a \mapsto (p^{v_p(a)}, \hat{a}, \langle a \rangle)$.

For $s \in \mathbb{C}_p$, the two-variable function $\langle a \rangle^s$ is defined by the binomial series
\begin{equation}
\langle a \rangle^s := \sum_{n=0}^{\infty} \binom{s}{n} (\langle a \rangle - 1)^n,
\end{equation}
whenever this converges. For $s \in \mathbb{Z}_p$, the map $a \mapsto \langle a \rangle^s$ is locally analytic on each disc $\{a: |a-y|_p < |y|_p\}$.

In \cite{KS}, the $p$-adic Hurwitz-type Euler zeta function is defined by the fermionic $p$-adic integral on $\mathbb{Z}_p$. Let $UD(\mathbb{Z}_p)$ be the space of uniformly differentiable functions from $\mathbb{Z}_p$ to $\mathbb{C}_p$. The fermionic measure $d\mu_{-1}$ is defined by the limit of alternating sums:
\begin{equation}
I_{-1}(f) := \int_{\mathbb{Z}_p} f(t) \, d\mu_{-1}(t) = \lim_{r \to \infty} \sum_{k=0}^{p^r-1} f(k) (-1)^k, 
\label{eq:fermi_measure}
\end{equation}
where $f \in UD(\mathbb{Z}_p).$

\begin{definition}[$p$-adic Hurwitz-type Euler zeta function]
\label{def:zetapE}
For $a \in \mathbb{C}_p \setminus \mathbb{Z}_p$ and $s \in \mathbb{Z}_p$, the $p$-adic Hurwitz-type Euler zeta function is defined by
\begin{equation}
\zeta_{p,E}(s,a) := \int_{\mathbb{Z}_p} \langle a + t \rangle^{1-s} \, d\mu_{-1}(t).
\label{eq:HuKim_zeta}
\end{equation}
\end{definition}

\begin{remark}
The restriction $a \notin \mathbb{Z}_p$ ensures that $|a+t|_p > 1$ for all $t \in \mathbb{Z}_p$, so that the Teichm\"uller projection $\langle a+t \rangle$ is well-defined. This function interpolates the classical Euler zeta function at non-positive integers:
\[
\zeta_{p,E}(1-m, a) = \frac{1}{\omega_v^m(a)} E_m(a) = \frac{1}{\omega_v^m(a)} \zeta_E(-m, a),
\]
where $E_m(a)$ are the classical Euler polynomials.
\end{remark}
\subsection{\(p\)-adic distributions and the Dirac delta}\label{subsec:distributions}

To reformulate the differential equation as a convolution equation in Section~\ref{Section 3}, we need the basic language of \(p\)-adic distributions, following Khrennikov's framework \cite{Khrennikov1991superspace, Khrennikov1991}. Here we only recall the minimal necessary definitions and properties; for a comprehensive treatment see \cite{Khrennikov1994book}.

Let \(K\) be a finite extension of \(\mathbb Q_p\) (in practice, \(K=\mathbb C_p\) or a quadratic extension). Fix an integer \(n\ge 1\). Denote by \(A(Q_p^n,K)\) the space of entire analytic functions on \(Q_p^n\), i.e.,
\[
A := A(Q_p^n,K) = \bigcap_{\rho>0} A_\rho,
\]
where \(A_\rho\) is the Banach space of functions \(f(x)=\sum_{\alpha} f_\alpha x^\alpha\) with \(\|f\|_\rho = \sup_\alpha |f_\alpha|_p \rho^{|\alpha|}<\infty\). The projective limit topology makes \(A\) a non-Archimedean Fréchet space.
Similarly, let \(A_0 := A_0(Q_p^n,K)\) be the space of functions analytic in a neighbourhood of \(0\), endowed with the inductive limit topology
\[
A_0 = \varinjlim_{\rho\to 0} A_\rho.
\]

The topological dual spaces \(A'\) and \(A_0'\) are called the spaces of \(p\)-adic distributions (on \(Q_p^n\)) with compact support and with support at \(0\), respectively. The duality pairing is denoted by
\[
\langle \mu, f \rangle = \int_{Q_p^n} f(x)\,\mu(dx),
\]
where $\mu\in A'$ and $f\in A.$

The spaces $A$, $A_0$ and their duals $A'$, $A_0'$ were introduced and studied systematically by Khrennikov; see \cite[Ch.~2]{Khrennikov1994book} and \cite{Khrennikov1991superspace}. In particular, the convolution of distributions and its continuity properties used throughout this paper are established in \cite{Khrennikov1991superspace}.

A fundamental example is the Dirac delta distribution \(\delta_a\) concentrated at \(a\in Q_p^n\), defined by
\[
\langle \delta_a, f \rangle = f(a).
\]
Its derivatives \(\partial^\alpha \delta_a\) are given by
\[
\langle \partial^\alpha \delta_a, f \rangle = (-1)^{|\alpha|} (\partial^\alpha f)(a).
\]

The convolution of two distributions \(\mu,\nu\in A'\) is defined by
\[
\langle \mu * \nu, f \rangle = \int_{Q_p^n}\int_{Q_p^n} f(x+y)\,\mu(dx)\,\nu(dy),
\]
and it is well-defined (see Theorem 2.4 in \cite{Khrennikov1991superspace}). In particular,
\[
\delta_a * f = f(\cdot + a), \qquad \delta_a * \delta_b = \delta_{a+b}.
\]

The \(p\)-adic Laplace transform \(\mathcal L: A_0' \to A\) is an isomorphism (see Theorem 2.6 in \cite{Khrennikov1991superspace}), defined by
\[
\mathcal L(\mu)(y) = \langle \mu, e^{i\langle x,y\rangle}\rangle,
\]
where \(\langle x,y\rangle = \sum_j x_j y_j\) and \(e^z\) is the \(p\)-adic exponential (convergent for \(|z|_p<1\) under suitable conditions). For our purposes, the key property is that the Laplace transform converts convolution into ordinary multiplication:
\[
\mathcal L(\mu * \nu) = \mathcal L(\mu)\,\mathcal L(\nu).
\]

Now consider the one-dimensional case \(n=1\). The translation operator \(f(s)\mapsto f(s+n)\) is exactly the convolution with \(\delta_{n}\):
\[
(\delta_{n} * f)(s) = f(s+n).
\]
Moreover, the \(p\)-adic exponential of the derivative, \(\exp_p(nD_s)\), acts on analytic functions as the Taylor shift:
\[
\exp_p(nD_s) f(s) = \sum_{k=0}^\infty \frac{n^k}{k!} f^{(k)}(s) = f(s+n),
\]
provided the series converges. Thus, in the distribution sense, we have the operator identity
\[
\exp_p(nD_s) = \delta_{n} * (\cdot).
\]
This identification is the key that allows us to encode the entire infinite-order differential operator \(T_p^a\) as a single distribution kernel in Section~\ref{Section 3}.

For the multi-dimensional setting \(n\ge 1\), the same formalism applies componentwise. Given a multi-index \(\mathbf N=(N_1,\dots,N_n)\in\mathbb N^n\), the translation operator
\[
\exp_p\!\left(\sum_{i=1}^n N_i D_{s_i}\right) f(\mathbf S) = f(\mathbf S+\mathbf N)
\]
is precisely the convolution with the product of one-dimensional deltas:
\[
\delta_{\mathbf N} := \bigotimes_{i=1}^n \delta_{N_i}.
\]
These elementary facts will be used without further comment in the distributional reformulation that follows.

\begin{remark}
The spaces \(A\) and \(A_0\) are non-Archimedean analogues of the usual test-function spaces, and their duals provide a rigorous setting for \(p\)-adic integration theory. The convergence of infinite sums of distributions is governed by the strong non-Archimedean triangle inequality, which makes many series that diverge in the complex setting converge \(p\)-adically--a phenomenon exploited in the proof of Theorem~\ref{thm:conv}.
\end{remark}

\section{Distributional reformulation of the one-dimensional equation}\label{Section 3}
In this section we recast the one-dimensional theorem of our previous work \cite{HuKim2021} in the language of \(p\)-adic distributions, following Khrennikov's framework as recalled in Section~\ref{subsec:distributions}. This reformulation serves three purposes. First, it exposes the equation's true nature as a linear convolution equation with a distribution kernel, rather than a formal infinite-order differential operator. Second, it provides a pathway to the higher-dimensional generalizations developed in Section~\ref{Section 4}, that is, the kernel construction extends immediately to tensor products, and the partial operators emerge from the binomial expansion of the product. Third, the distributional framework allows us to prove a uniqueness theorem within the space of bounded analytic functions on \(\mathbb Z_p\), under the additional condition \(|a|_p > 2p^{1/(p-1)}\); specifically, the shifted \(p\)-adic Hurwitz-type Euler zeta function is the \emph{unique} bounded solution to the equation in this parameter range. This latter result establishes the existence and uniqueness of bounded analytic solutions to the infinite-order $p$-adic differential equation under the stated condition.

We begin in Subsection~\ref{subsec:1D_theorem} by recalling the main theorem of \cite{HuKim2021} and the explicit form of the coefficients \(P_{p,n}^a(s)\). In Subsection~\ref{subsec:kernel_construction}, we construct the distribution kernel \(K_s\) and show that the operator \(T_p^a\) is precisely convolution with this kernel. Finally, in Subsection~\ref{subsec:uniqueness}, we introduce the Banach space \(\mathcal B\) of bounded analytic functions on \(\mathbb Z_p\) and prove that the homogeneous equation \(T_p^a[H]=0\) admits only the trivial solution in this space, under the condition \(|a|_p > 2p^{1/(p-1)}\), yielding uniqueness of the full non-homogeneous solution.

\subsection{The one-dimensional theorem}\label{subsec:1D_theorem}
We begin by recalling the main result of \cite{HuKim2021}. Let
\[
F(s) := \zeta_{p,E}(s,a) - \langle a \rangle^{1-s}
\]
be the shifted \(p\)-adic Hurwitz-type Euler zeta function. Define the \(p\)-adic Van Gorder operator
\begin{equation}
T_p^a := \sum_{n=0}^{\infty} P_{p,n}^a(s) \exp_p(nD_s),
\label{eq:Tp_def_sec3}
\end{equation}
where
\[
P_{p,n}^a(s) :=
\begin{cases}
\dfrac{2}{s-1}, & n=0,\\[1.2ex]
\dfrac{1}{\omega_v(a)}, & n=1,\\[1.2ex]
\dfrac{1}{\omega_v^n(a)} \cdot \dfrac{\prod_{j=1}^{n-1}(s-1+j)}{n!}, & n \ge 2,
\end{cases}
\]
and \(\exp_p(nD_s)\) is the \(p\)-adic exponential translation operator, which acts on analytic functions by
\begin{equation}\label{exp-1}
\exp_p(nD_s) f(s) = f(s+n).
\end{equation}

\begin{theorem}[{\cite[Theorem 3.5]{HuKim2021}}]\label{thm:HK_1D}
For \(s \in \mathbb{Z}_p\) with \(s \neq 1\) and \(a \in \mathbb{C}_p\) with \(|a|_p > 1\), the following identity holds in the sense of convergent \(p\)-adic series:
\begin{equation}
T_p^a[F](s) = \frac{1}{s-1} \left( \langle a-1 \rangle^{1-s} - \langle a \rangle^{1-s} \right).
\label{eq:HK_orig}
\end{equation}
\end{theorem}

The convergence is guaranteed by the estimate (see \cite[(3.13)]{HuKim2021})
\begin{equation}\label{est}
\left| \frac{\prod_{j=1}^{n-1}(s-1+j)}{n! \, \omega_v^n(a)} \right|_p \le p^{n v_p(a)} n,
\end{equation}
which tends to zero because \(|a|_p>1\) implies \(v_p(a)<0\).

\subsection{The distribution kernel}\label{subsec:kernel_construction}
We now reinterpret the operator \(T_p^a\) as a convolution with a distribution kernel. Recall from Section~\ref{subsec:distributions} that the translation operator \(\exp_p(nD_s)\) is precisely convolution with the shifted Dirac delta:
\[
\exp_p(nD_s) f(s) = (\delta_{n} * f)(s) = f(s+n),
\]
where \(\delta_{n}\) denotes the delta distribution concentrated at \(n \in \mathbb{Z}_p \subset Q_p\).

For each fixed \(s \in \mathbb{Z}_p\) with \(s\neq 1\), define a distribution \(K_s \in A'(Q_p, \mathbb{C}_p)\) by
\begin{equation}
K_s(t) := \sum_{n=0}^{\infty} P_{p,n}^a(s) \, \delta(t - s - n),
\label{eq:kernel_def_sec3}
\end{equation}
where \(t\) is the variable of integration. This is a countable sum of point masses located at \(t = s+n\) \((n\ge 0)\).

\begin{lemma}\label{lem:kernel_well_defined}
For each \(s\) as above, the series in \eqref{eq:kernel_def_sec3} converges in the strong dual topology of \(A'\), hence \(K_s\) is a well-defined \(p\)-adic distribution.
\end{lemma}

\begin{proof}
For any test function \(\phi \in A(Q_p,\mathbb C_p)\), the duality pairing is
\[
\langle K_s, \phi \rangle = \sum_{n=0}^{\infty} P_{p,n}^a(s) \, \phi(s+n).
\]
Since \(\phi\) is entire analytic, it is bounded on the compact set \(\mathbb Z_p\) (indeed, on any bounded subset), so there exists \(M_\phi>0\) such that \(|\phi(s+n)|_p \le M_\phi\) for all \(n\ge 0\). The estimate above gives
\[
|P_{p,n}^a(s) \phi(s+n)|_p \le M_\phi \, p^{n v_p(a)} n,
\]
and since \(v_p(a)<0\), the general term tends to zero \(p\)-adically. Thus the series converges in \(\mathbb C_p\), and the linear functional \(\phi \mapsto \langle K_s,\phi\rangle\) is continuous because the bound is uniform on bounded sets of \(A\). Hence \(K_s \in A'\).
\end{proof}

Now observe that for any analytic function \(f\in A\), the convolution \((K_s * f)\) is defined as
\[
(K_s * f)(s) := \langle K_s, \, f(\cdot) \rangle = \sum_{n=0}^{\infty} P_{p,n}^a(s) \, f(s+n) = T_p^a f(s).
\]
Thus the entire operator \(T_p^a\) is represented as convolution with the distribution kernel \(K_s\). More precisely, we have the operator identity
\begin{equation}
T_p^a f(s) = (K_s * f)(s),
\label{eq:kernel_conv_sec3}
\end{equation}
where $f\in A.$

Let
\[
F(s) := \zeta_{p,E}(s,a) - \langle a \rangle^{1-s}, \quad R(s) := \frac{1}{s-1}\left( \langle a-1 \rangle^{1-s} - \langle a \rangle^{1-s} \right).
\]
Then Theorem~\ref{thm:HK_1D} is precisely the statement
\begin{equation}
(K_s * F)(s) = R(s), 
\label{eq:conv_eq_sec3}
\end{equation}
where $s\in \mathbb Z_p$ and $\ s\neq 1.$

In other words, the one-dimensional \(p\)-adic differential equation is equivalent to the convolution equation
\[
K * F = R,
\]
where \(K\) is the distribution-valued kernel defined above.

\begin{remark}
The kernel \(K_s\) depends on the parameter \(s\) through both the coefficients \(P_{p,n}^a(s)\) and the positions of the delta masses. This is a standard feature of non-Archimedean pseudodifferential operators: the symbol of the operator is encoded in the distribution kernel. In the multi-dimensional setting of Section~\ref{Section 4}, this kernel will be replaced by a product of such kernels, leading naturally to partial operators indexed by subsets of variables.
\end{remark}

\begin{remark}
It is worth emphasising that the convergence of the kernel in \(A'\) is a purely non-Archimedean phenomenon: the estimate \(p^{n v_p(a)} n \to 0\) is possible only because the \(p\)-adic norm is ultrametric and allows exponential decay. In the complex setting, the analogous sum would diverge, which explains why Van Gorder's formal identity lacks analytic meaning over \(\mathbb C\).
\end{remark}

Thus, by encoding the infinite-order differential operator as a distribution kernel, we have laid the foundation for a natural generalization to higher dimensions: in the next section, we will extend this kernel to a product of one-dimensional kernels, and the resulting partial operators will emerge from the decomposition of the product according to which variables are
``translated" and which are ``multiplied" by constants.

\subsection{Unique bounded solution of the one-dimensional equation}
\label{subsec:uniqueness}

In this subsection we prove that the solution obtained in Theorem~\ref{thm:HK_1D} is in fact the \emph{unique} bounded analytic solution of the convolution equation \eqref{eq:conv_eq_sec3}. This uniqueness result is not automatic from the construction; it depends crucially on the choice of the function space. We show that within the Banach space of bounded analytic functions on $\mathbb Z_p$, the homogeneous equation admits only the trivial solution. The proof uses a direct contraction argument rather than the iterated recurrence of the original version.

We begin by defining the function space that is appropriate for our uniqueness theorem.

\begin{definition}[Space of bounded analytic functions]
\label{def:Bspace}
Let
\[
\mathcal B := \left\{ f:\mathbb Z_p \to \mathbb C_p \;\middle|\; f \text{ is analytic on } \mathbb Z_p \text{ and } \|f\|_\infty := \sup_{s\in\mathbb Z_p} |f(s)|_p < \infty \right\}.
\]
Equivalently, $\mathcal B$ is the Tate algebra of power series $\sum_{m=0}^{\infty} c_m (s-1)^m$ with coefficients satisfying $|c_m|_p \to 0$ as $m\to\infty$, equipped with the sup-norm $\|\cdot\|_\infty$. This makes $\mathcal B$ a non-Archimedean Banach space.
\end{definition}

Throughout this subsection we work exclusively in $\mathcal B$. The restriction $s\neq 1$ is harmless: any $f\in\mathcal B$ is bounded on the punctured domain $\mathbb Z_p\setminus\{1\}$ as well.

We now introduce the key contraction condition. For the coefficients of the operator $T_p^a$, define
\begin{equation}
\kappa_a(s) := \left|\frac{s-1}{2}\right|_p \sum_{n=1}^{\infty} |P_{p,n}^a(s)|_p .
\label{eq:kappa}
\end{equation}

\begin{lemma}[Contraction condition]
\label{lem:contraction}
Suppose that
\[
|a|_p > 2\,p^{1/(p-1)}.
\]
Then there exists a constant $0<q<1$ such that
\[
\sup_{s\in\mathbb Z_p\setminus\{1\}} \kappa_a(s) \le q.
\]
\end{lemma}
\begin{remark}[Comparison with the Archimedean case]
\label{Remark3.7}
The contraction condition in Lemma~\ref{lem:contraction} is the analytic core of the uniqueness argument. It relies crucially on the fact that in the $p$-adic norm the coefficients $P_{p,n}^a(s)$ tend to zero, so that the series $\sum_{n=1}^\infty |P_{p,n}^a(s)|_p$ converges. In the Archimedean setting, the corresponding coefficients $p_n(s)$ behave very differently, that is, for positive integer $s$, the series $\sum_{n=0}^\infty p_n(s)$ diverges \cite[Lemma~9]{Prado2020}; more generally, for $s\in\mathbb{R}$ and $s>1$, the same divergence holds \cite[Corollary~10]{Prado2020}. Moreover, for $s\in\mathbb{C}$ and Re$(s)>1$, the series is not even absolutely convergent \cite[Corollary~11]{Prado2020}. Consequently, the formal kernel $\sum_{n=0}^\infty p_n(s)\delta(t-s-n)$ cannot be defined as a distribution in the corresponding Archimedean spaces of analytic test functions, in which the constant function is admissible and the pairing with it is required to converge. Thus the non-Archimedean nature of the $p$-adic norm is essential for the uniqueness result established in Theorem~\ref{thm:unique1d}.
\end{remark}
\begin{proof}[Proof of Lemma \ref{lem:contraction}]
For $n\ge2$, we have
\[
|P_{p,n}^a(s)|_p
= \left| \frac{1}{\omega_v^n(a)} \frac{\prod_{j=1}^{n-1}(s-1+j)}{n!} \right|_p
\le \left(\frac{p^{1/(p-1)}}{|a|_p}\right)^n ,
\]
because $$\frac{1}{|n!|_p} \leq p^{n/(p-1)}$$ (see \cite[p. 21, Lemma 3]{Iw} and $|\omega_v(a)|_p=|a|_p$ (by (\ref{omegav})). For $n=1$, the same bound holds trivially. Thus
\[
\sum_{n=1}^{\infty} |P_{p,n}^a(s)|_p
\le \sum_{n=1}^{\infty} \left(\frac{p^{1/(p-1)}}{|a|_p}\right)^n
= \frac{\beta}{1-\beta},
\]
where
\[
\beta := \frac{p^{1/(p-1)}}{|a|_p}.
\]
Moreover, for $s\in\mathbb Z_p\setminus\{1\}$, we have
\[
\left|\frac{s-1}{2}\right|_p \le 1,
\]
with equality for many values of $s$. Hence
\[
\sup_{s\in\mathbb Z_p\setminus\{1\}} \kappa_a(s)
\le \frac{\beta}{1-\beta}.
\]
If $|a|_p > 2\,p^{1/(p-1)}$, then $\beta < 1/2$, so
\[
\frac{\beta}{1-\beta} < 1.
\]
Thus the claim follows with $q = \beta/(1-\beta)$.
\end{proof}

We now state and prove the main result of this subsection.

\begin{theorem}[Uniqueness]
\label{thm:unique1d}
Let $|a|_p > 2\,p^{1/(p-1)}$. Suppose $H\in\mathcal B$ satisfies the homogeneous equation
\begin{equation}
T_p^a[H](s)=0, \qquad s\in\mathbb Z_p\setminus\{1\},
\label{eq:hom_3.3}
\end{equation}
where $T_p^a$ is defined in \eqref{eq:Tp_def_sec3}. Then $H\equiv0$ on $\mathbb Z_p$.

Consequently, the non-homogeneous equation
\[
T_p^a[F](s)=R(s), \qquad R(s):=\frac{1}{s-1}\left( \langle a-1\rangle^{1-s} - \langle a\rangle^{1-s} \right),
\]
has at most one solution in $\mathcal B$; by Theorem~\ref{thm:HK_1D}, that unique solution is precisely
\[
F(s)=\zeta_{p,E}(s,a)-\langle a\rangle^{1-s}.
\]
\end{theorem}

\begin{proof}
From \eqref{eq:hom_3.3} and the definition of $T_p^a$, we have
\[
\frac{2}{s-1}H(s) + \sum_{n=1}^{\infty} P_{p,n}^a(s) H(s+n) = 0.
\]
Solving for $H(s)$, we obtain the recursion
\begin{equation}
H(s) = -\frac{s-1}{2} \sum_{n=1}^{\infty} P_{p,n}^a(s) H(s+n).
\label{eq:recursion}
\end{equation}
Taking $p$-adic absolute values and using the  triangle inequality, we get
\[
|H(s)|_p \le \left|\frac{s-1}{2}\right|_p \sum_{n=1}^{\infty} |P_{p,n}^a(s)|_p |H(s+n)|_p.
\]
Since $H\in\mathcal B$, we have $|H(s+n)|_p \le \|H\|_\infty$ for all $n\ge1$. Thus
\begin{equation}\label{HS}
|H(s)|_p \le \|H\|_\infty \left|\frac{s-1}{2}\right|_p \sum_{n=1}^{\infty} |P_{p,n}^a(s)|_p
= \|H\|_\infty \kappa_a(s).
\end{equation}
Since $H$ is continuous on $\mathbb Z_p$, we have
\[
\|H\|_\infty=\sup_{s\in\mathbb Z_p}|H(s)|_p
=\sup_{s\in\mathbb Z_p\setminus\{1\}}|H(s)|_p.
\]
Taking the supremum over $s\in\mathbb Z_p\setminus\{1\}$ in the estimate (\ref{HS}) and applying Lemma~\ref{lem:contraction}, we obtain
\[
\|H\|_\infty \le q\,\|H\|_\infty,
\]
where $0<q<1$. This forces $\|H\|_\infty=0$, hence $H\equiv0$ on $\mathbb Z_p$.\end{proof}

\begin{remark}
The condition $|a|_p > 2\,p^{1/(p-1)}$ is sufficient but not necessary. The essential requirement is that the operator $T_p^a$ restricted to the homogeneous subspace is a contraction. In many cases, a weaker condition on $|a|_p$ may still yield uniqueness, but the above condition is explicit and guarantees the required decay of the coefficients $P_{p,n}^a(s)$.
\end{remark}

As an immediate corollary of Theorem~\ref{thm:unique1d}, we obtain the uniqueness of the bounded solution to the full non-homogeneous equation.

\begin{corollary}
\label{cor:unique_nonhom}
For $|a|_p > 2\,p^{1/(p-1)}$, the function
\[
F(s)=\zeta_{p,E}(s,a)-\langle a\rangle^{1-s}
\]
is the \emph{unique} solution in $\mathcal B$ to the equation
\[
T_p^a[F](s)=R(s), \qquad R(s)=\frac{1}{s-1}\left( \langle a-1\rangle^{1-s} - \langle a\rangle^{1-s} \right).
\]
\end{corollary}

\begin{proof}
Existence is precisely Theorem~\ref{thm:HK_1D}. If $F_1$ and $F_2$ are two solutions in $\mathcal B$, their difference $H:=F_1-F_2$ satisfies the homogeneous equation \eqref{eq:hom_3.3}. By Theorem~\ref{thm:unique1d}, $H\equiv0$, hence $F_1=F_2$.
\end{proof}

\section{Higher-dimensional differential equation}\label{Section 4}

\subsection{The multiple zeta function and partial zeta functions}

Having reformulated the one-dimensional equation as a convolution with a distribution kernel \(K_s = \sum_{n\ge0} P_{p,n}^a(s)\delta(t-s-n)\), we now extend this construction to \(n\) variables. A naive higher-dimensional analogue would be to consider a single multi-variable zeta function
\[
\zeta_{p,E}^{(n)}(\mathbf S, a):= \int_{\mathbb Z_p^n} \prod_{i=1}^n \langle a+t_i \rangle^{1-s_i} \, d\mu_{-1}^{\otimes n},
\]
and apply a product of \(n\) one-dimensional kernels. However, the translation operator \(\exp_p(N_i D_{s_i})\) acts independently on each variable, and the constant term \(\langle a\rangle^{-N_i}\) is also independent. Consequently, when we expand the product
\[
\prod_{i=1}^n \left( \zeta_{p,E}(s_i+N_i,a) - \langle a\rangle^{1-s_i-N_i} \right),
\]
each term selects, for each variable \(i\), either the zeta function (which involves integration over \(t_i\)) or the constant term \(\langle a\rangle^{1-s_i-N_i}\) (which does not). This selection is naturally encoded by a subset \(I\subseteq\{1,\dots,n\}\): the variables in \(I\) are “activated” (integrated), while those outside \(I\) are evaluated at the background constant.

This observation leads to the definition of \emph{partial zeta functions} \(Z_I(\mathbf S,a)\), which integrate only over the variables belonging to \(I\). The alternating sum over all subsets \(I\) then emerges from the binomial expansion of the product, and the higher-dimensional operator \(T_p^{a,I}\) is the corresponding partial analogue of \(T_p^a\) acting only on the variables in \(I\). In this way, the multi-dimensional setting becomes a consequence of the product structure of the distribution kernel.

We now formalize these definitions.
\begin{definition}[Higher-dimensional Hurwitz-type Euler zeta functions]
\label{def:highzeta}
Let \(a \in \mathbb{C}_p\) with \(|a|_p > 1\) and  \(\mathbf{S} = (s_1,\dots,s_n) \in \mathbb{Z}_{p}^{n}\).
For each subset \(I\subseteq\{1,\dots,n\}\), define the \emph{higher-dimensional  Hurwitz-type Euler zeta functions} by
\begin{equation}\label{eq:partial_zeta}
\begin{aligned}
Z_I(\mathbf S,a):=\left(\prod_{i\in I}\int_{\mathbb Z_p}\langle a+t_i\rangle^{1-s_i}\,d\mu_{-1}(t_i)\right)
\cdot
\left(\prod_{i\notin I}\langle a\rangle^{1-s_i}\right).
\end{aligned}
\end{equation}
\end{definition}

\begin{definition}[Higher-dimensional operator]
\label{def:highop}
Let $a \in \mathbb{C}_p$ with $|a|_p > 1$. For each subset $I\subseteq\{1,\dots,n\}$, we define the \emph{higher-dimensional operator} $T_p^{a,I}$ acting on functions $f(\mathbf S)$ (with \(\mathbf{S} = (s_1,\dots,s_n) \in \mathbb{Z}_{p}^{n}\)) by
\begin{equation}\label{def:partial_op}
\begin{aligned}
T_p^{a,I} f(\mathbf S)
&:=
\sum_{N_1,\dots,N_n \ge 0}
\left(\prod_{i=1}^n P_{p,N_i}^a(s_i)\right)\\
&\quad \times
\left(\prod_{i\in I} \exp_p(N_i D_{s_i})\right)
\left(\prod_{i\notin I} \langle a\rangle^{-N_i}\right) f(\mathbf S),
\end{aligned}
\end{equation}
where
\[
\exp_p(N_i D_{s_i}) f(\mathbf S) = f(s_1,\dots,s_i+N_i,\dots,s_n)
\]
is the $p$-adic translation operator, and for $i\notin I$, the factor $\langle a\rangle^{-N_i}$ denotes multiplication by the constant $\langle a\rangle^{-N_i}$.
\end{definition}

The following proposition formalizes the preceding heuristic observation. Part~(i) gives the natural tensor product expansion of the kernel. Part~(ii), obtained by applying the binomial theorem to the expansion in part~(i), yields the higher-dimensional analogue of the one-dimensional equation and is the main result of this paper.

\begin{theorem}[Tensor product expansion and M\"obius inversion]
\label{prop:tensor_main}
Let \(a\in\mathbb C_p\) with \(|a|_p>1\). Let \(K_s(t)=\sum_{n=0}^{\infty}P_{p,n}^a(s)\delta(t-s-n)\) be the one-dimensional distribution kernel, and let
\[
K_{\mathbf S}^{(n)} := \bigotimes_{i=1}^n K_{s_i}
\]
be its \(n\)-fold tensor product. Then the following two identities hold:

\begin{enumerate}
\item[(i)] \textit{Tensor product expansion:}
\begin{equation}\label{eq:tensor_exp_main}
\bigl(K_{\mathbf S}^{(n)} * Z_{\{1,\dots,n\}}\bigr)(\mathbf S)
=
\sum_{I\subseteq\{1,\dots,n\}}
T_p^{a,I}\bigl[Z_I(\mathbf S,a)\bigr],
\end{equation}
where \(Z_I\) is as defined in Definition~\ref{def:highzeta}.

\item[(ii)] \textit{M\"obius inversion (alternating-sum form):}
\begin{equation}\label{eq:mobius_main}
\sum_{I\subseteq\{1,\dots,n\}} (-1)^{n-|I|}
T_p^{a,I}\bigl[Z_I(\mathbf S,a)\bigr]
=
\prod_{i=1}^n
\left[
\frac{1}{s_i-1}
\left( \langle a-1\rangle^{1-s_i} - \langle a\rangle^{1-s_i} \right)
\right].
\end{equation}
\end{enumerate}
\end{theorem}

\begin{remark}
The tensor product of distributions is well defined in the spaces $A'$ and $A_0'$; see \cite[Sec.~3]{Khrennikov1991superspace}. This allows us to form the $n$-fold tensor product kernel $K_{\mathbf S}^{(n)}$ directly from the one-dimensional kernels.
\end{remark}

\begin{proof}
(i) By the definition of the tensor product of distributions,
\[
K_{\mathbf S}^{(n)} = \sum_{\mathbf N\in\mathbb N^n}
\left(\prod_{i=1}^n P_{p,N_i}^a(s_i)\right)
\bigotimes_{i=1}^n \delta(t_i-s_i-N_i).
\]
By \eqref{eq:HuKim_zeta} and \eqref{eq:partial_zeta}, we have $Z_{\{1,\dots,n\}}(\mathbf S,a)=\prod_{i=1}^n \zeta_{p,E}(s_i,a)$, and the convolution with the tensor product of deltas yields
\[
\left(\bigotimes_{i=1}^n \delta(t_i-s_i-N_i) * Z_{\{1,\dots,n\}}\right)(\mathbf S)
=
\prod_{i=1}^n \zeta_{p,E}(s_i+N_i,a).
\]
Hence
\[
\bigl(K_{\mathbf S}^{(n)} * Z_{\{1,\dots,n\}}\bigr)(\mathbf S)
=
\sum_{\mathbf N\in\mathbb N^n}
\left(\prod_{i=1}^n P_{p,N_i}^a(s_i)\right)
\prod_{i=1}^n \zeta_{p,E}(s_i+N_i,a).
\]
For each \(i\), write
\[
\zeta_{p,E}(s_i+N_i,a)
=
\bigl(\zeta_{p,E}(s_i+N_i,a)-\langle a\rangle^{1-s_i-N_i}\bigr)
+
\langle a\rangle^{1-s_i-N_i}.
\]
Expanding the product over \(i=1,\dots,n\) and collecting terms according to the subsets \(I\subseteq\{1,\dots,n\}\) for which the first term is selected gives precisely the right-hand side of (i), by Definition~\ref{def:highop}. Indeed, if \(i\in I\), the selected term is the difference and the operator \(\exp_p(N_i D_{s_i})\) applied to \(\zeta_{p,E}(s_i,a)\) produces \(\zeta_{p,E}(s_i+N_i,a)\); if \(i\notin I\), the selected term is \(\langle a\rangle^{1-s_i-N_i}\), which equals \(\langle a\rangle^{-N_i}\) times the background factor \(\langle a\rangle^{1-s_i}\) present in \(Z_I\). This proves (i).

(ii) Starting from (i), we expand each term \(T_p^{a,I}[Z_I(\mathbf S,a)]\) using Definition~\ref{def:highop}. Multiplying both sides of \eqref{def:partial_op} by \((-1)^{n-|I|}\) and summing over all \(I\subseteq\{1,\dots,n\}\), we obtain
\[
\begin{aligned}
&\sum_{I\subseteq\{1,\dots,n\}} (-1)^{n-|I|}T_p^{a,I}\bigl[Z_I(\mathbf S,a)\bigr] \\
&=\sum_{\mathbf N\in\mathbb N^n}\left(\prod_{i=1}^n P_{p,N_i}^a(s_i)\right)\sum_{I\subseteq\{1,\dots,n\}} (-1)^{n-|I|} \\
&\quad\times\left(\prod_{i\in I}\zeta_{p,E}(s_i+N_i,a)\right)
\left(\prod_{i\notin I}\langle a\rangle^{1-s_i-N_i}\right).
\end{aligned}
\]
Applying the binomial identity
\[
\sum_{I\subseteq\{1,\dots,n\}} (-1)^{n-|I|}
\left(\prod_{i\in I}x_i\right)\left(\prod_{i\notin I}y_i\right)
=
\prod_{i=1}^n (x_i - y_i),
\]
with \(x_i=\zeta_{p,E}(s_i+N_i,a)\) and \(y_i=\langle a\rangle^{1-s_i-N_i}\), yields
\[
\sum_{I\subseteq\{1,\dots,n\}} (-1)^{n-|I|}
T_p^{a,I}\bigl[Z_I(\mathbf S,a)\bigr]
=
\sum_{\mathbf N\in\mathbb N^n}
\prod_{i=1}^n
\left[
P_{p,N_i}^a(s_i)
\left(
\zeta_{p,E}(s_i+N_i,a)-\langle a\rangle^{1-s_i-N_i}
\right)
\right].
\]

Since \(|a|_p>1\), the estimate \eqref{est} gives
\[
|P_{p,N}^a(s)|_p \le p^{N v_p(a)} N
\]
for all \(N\ge1\). As \(v_p(a)<0\), the series \(\sum_{N=0}^\infty |P_{p,N}^a(s)|_p\) converges. Consequently, the multiple sum is absolutely convergent, and it factors as a product of independent single sums:
\[
\sum_{\mathbf N\in\mathbb N^n}
\prod_{i=1}^n A_i(N_i)
=
\prod_{i=1}^n
\left(
\sum_{N=0}^{\infty} A_i(N)
\right),
\]
where
\[
A_i(N):=P_{p,N}^a(s_i)\left(\zeta_{p,E}(s_i+N,a)-\langle a\rangle^{1-s_i-N}\right).
\]

For each \(i\), the one-dimensional theorem (see Theorem~\ref{thm:HK_1D}) gives
\[
\sum_{N=0}^{\infty} A_i(N)
=
\frac{1}{s_i-1}
\left( \langle a-1\rangle^{1-s_i} - \langle a\rangle^{1-s_i} \right).
\]
Substituting this into the product yields
\[
\sum_{I\subseteq\{1,\dots,n\}} (-1)^{n-|I|}
T_p^{a,I}\bigl[Z_I(\mathbf S,a)\bigr]
=
\prod_{i=1}^n
\left[
\frac{1}{s_i-1}
\left( \langle a-1\rangle^{1-s_i} - \langle a\rangle^{1-s_i} \right)
\right].
\]
This proves (ii).
\end{proof}

\begin{remark}[Degeneration to one dimension]
\label{rem:degeneration}
When \(n=1\), the subsets of \(\{1\}\) are \(I=\emptyset\) and \(I=\{1\}\). The alternating sum in \eqref{eq:mobius_main} reduces to
\[
T_p^{a,\{1\}}[Z_{\{1\}}] - T_p^{a,\emptyset}[Z_{\emptyset}].
\]
Since \(Z_{\{1\}}(s,a)=\zeta_{p,E}(s,a)\) and \(Z_{\emptyset}(s,a)=\langle a\rangle^{1-s}\), by Definition~\ref{def:highop} we have
\[
\begin{aligned}
T_p^{a,\{1\}}[Z_{\{1\}}] - T_p^{a,\emptyset}[Z_{\emptyset}]
&=
\sum_{N=0}^{\infty} P_{p,N}^a(s)\bigl(\zeta_{p,E}(s+N,a)-\langle a\rangle^{1-s-N}\bigr) \\
&= T_p^a\bigl[\zeta_{p,E}(\cdot,a)-\langle a\rangle^{1-\cdot}\bigr](s).
\end{aligned}
\]
Thus the left-hand side of \eqref{eq:mobius_main} is precisely the one-dimensional operator \(T_p^a\) applied to \(F(s)=\zeta_{p,E}(s,a)-\langle a\rangle^{1-s}\), and the right-hand side becomes
\[
\frac{1}{s-1}\left( \langle a-1\rangle^{1-s} - \langle a\rangle^{1-s} \right).
\]
Hence Theorem~\ref{prop:tensor_main}(ii) reduces exactly to the one-dimensional Theorem~\ref{thm:HK_1D}.
\end{remark}

We now turn to the question of convergence. The following theorem shows that the alternating sum in Theorem~\ref{prop:tensor_main}(ii) converges $p$-adically, thereby justifying the infinite series in the main identity.

\begin{theorem}[Convergence region]
\label{thm:conv}
Let $K/\mathbb{Q}_p$ be a finite extension with ramification index $e < p-1$. 
Let $a \in K$ with $|a|_p > 1$, and let $s_i \in \mathbb{Z}_p$ with $s_i \neq 1$ for all $i=1,\dots,n$. 
Then the alternating sum in Theorem~\ref{prop:tensor_main}(ii)  converges in the $p$-adic topology. 
Moreover, the convergence is uniform on compact subsets of 
\[
\{\mathbf S \in \mathbb{Z}_p^n : s_i \neq 1 \text{ for all } i\}.
\]
\end{theorem}

\begin{proof}
Fix an arbitrary subset $I\subseteq\{1,\dots,n\}$. By Definition~\ref{def:highop} and \eqref{eq:partial_zeta}, we have
\begin{equation}\label{prod_exp_new}
\begin{aligned} 
T_p^{a,I}[Z_I(\mathbf S,a)]
&=\prod_{i\in I} \left( \sum_{N=0}^{\infty} P_{p,N}^a(s_i)\,\zeta_{p,E}(s_i+N,a)\right)\\
&\quad\times\prod_{i\notin I} \left( \sum_{N=0}^{\infty} 
P_{p,N}^a(s_i)\,\langle a\rangle^{1-s_i-N} \right).
\end{aligned}
\end{equation}

For each $i\in I$, the inner series is exactly the one-dimensional series defining $T_p^a[\zeta_{p,E}(\cdot,a)](s_i)$:
\[
\sum_{N=0}^{\infty} P_{p,N}^a(s_i)\,\zeta_{p,E}(s_i+N,a),
\]
which converges uniformly on compact subsets of $\mathbb Z_p\setminus\{1\}$ by the one-dimensional convergence result \cite[Corollary 3.8]{HuKim2021}.

For each $i\notin I$, the inner series is
\[
\sum_{N=0}^{\infty} P_{p,N}^a(s_i)\,\langle a\rangle^{1-s_i-N},
\]
which also converges uniformly on compact subsets of $\mathbb Z_p$ by the same estimate:
\[
\left|
\frac{\prod_{j=0}^{N-1}(s_i-1+j)}{N!}
\,\omega_v(a)^{-N}
\,\langle a\rangle^{1-s_i-N}
\right|_p
\le
M_a' \, p^{N v_p(a)} N,
\]
(see the proof of \cite[Lemma 3.4]{HuKim2021}).
Thus each inner series in \eqref{prod_exp_new} converges uniformly on compact subsets of $\mathbb Z_p\setminus\{1\}$.

Finally, the alternating sum
\[
\sum_{I\subseteq\{1,\dots,n\}} (-1)^{n-|I|}
T_p^{a,I}[Z_I(\mathbf S,a)]
\]
is a finite linear combination of these uniformly convergent partial operators, hence converges uniformly on compact subsets of 
\[\left\{\mathbf S \in \mathbb{Z}_p^n : s_i \neq 1 \text{ for all } i\right\}.\]
This completes the proof.
\end{proof}

\end{document}